\documentclass[a4paper, 11pt]{amsart}

\usepackage{amssymb}
\usepackage{amsmath}
\usepackage{amsthm}
\usepackage{enumitem}
\usepackage{mathtools}
\usepackage{pxfonts}

\usepackage{dsfont}
\usepackage{stmaryrd}
\usepackage{thmtools, thm-restate}

\usepackage{verbatim}

\newtheorem{thm}{Theorem}[section]
\newtheorem{lemma}[thm]{Lemma}
\newtheorem{cor}[thm]{Corollary}
\newtheorem{prop}[thm]{Proposition}
\theoremstyle{definition}

\newtheorem{example}[thm]{Example}

\theoremstyle{remark}
\newtheorem{remark}[thm]{Remark}
\newtheorem{claim}[thm]{Claim}

\newtheorem{obs}[thm]{Observation}

\numberwithin{equation}{section}

\newcommand{\acts}{\curvearrowright}

\newcommand{\cU}{\mathcal{U}}
\newcommand{\cH}{\mathcal{H}}

\DeclareMathOperator{\Aut}{Aut}

\usepackage{hyperref}
\hypersetup{
	colorlinks=true, 
	linkcolor=blue,
	citecolor= blue,
}

\begin{document}
\title[Infinite volume ends]{Infinite volume ends of quotient graphs \\and homogeneous spaces}
\begin{abstract}
    We introduce the space of infinite volume ends of a locally compact second countable (lcsc) space that admits a Radon measure. In certain cases, this coincides with the classical space of ends. 
    
    Consider a discrete subgroup $\Gamma$ of a unimodular lcsc group $G$ that is not coamenable. Assume that $G$ has property (T) and the associated homogeneous space $G/\Gamma$ is equipped with the Haar measure. We demonstrate that if $G$ is path connected, then $G/\Gamma$ has exactly one infinite volume end. In a related vein, if $G$ acts transitively on a locally finite connected graph $X$ with compact open vertex stabilizers and the action of the subgroup $\Gamma$ is free, we show that $X/\Gamma$ has exactly one end. We also obtain identical results for certain discrete subgroups $\Gamma$ of nonamenable product groups $G$.

    These results can be applied to understand ends of Schreier graphs and infinite volume ends of quotients of symmetric spaces of noncompact type. For instance, for symmetric spaces $X$ of noncompact type without real or complex hyperbolic factors, every infinite-covolume quotient $\Gamma\backslash X$ has exactly one end of infinite Riemannian volume.
\end{abstract}
\author[Konrad Wr\'{o}bel]{Konrad Wr\'{o}bel}
\address[Konrad Wr\'{o}bel]{Department of Mathematics, Jagiellonian University, Krak\'ow, Poland}
\email{konrad1.wrobel@uj.edu.pl}

\maketitle

\section{Introduction}

The study of ends of topological spaces has long been central in geometric group theory since Freudenthal \cite{Freudenthal1931}, where he introduced the notion and showed connected locally compact groups have at most two ends. The work of Freudenthal \cite{Freudenthal:2Ends}, and independently of Zippin \cite{Zippin:2Ends} and Iwasawa \cite{Iwasawa:2Ends}, then characterized those connected locally compact groups with two ends as those that are direct products of a compact group with the reals. 

For finitely generated discrete groups, the number of ends of their Cayley graphs is an important topological invariant, reflecting their algebraic structure. Hopf \cite{Hopf1944} showed these graphs can have 0, 1, 2, or infinitely many ends, depending on the group's properties, and that finite groups are those with no ends, while the groups with two ends are precisely those that are virtually $\mathbf{Z}$. The landmark articles of Stallings \cite{Stallings68,Stallings71} classified finitely generated groups with more than one end by showing they are exactly the groups that  split as either an amalgamated free product or HNN extension over a finite group. 
It is well known that infinite discrete groups with property (T) and products of infinite finitely generated groups have exactly one end.

Later developments extended these ideas to quotient spaces and graphs, reflecting more intricate algebraic and geometric structures. For a closed subgroup $H\leq G$ of a locally compact Hausdorff group, Houghton \cite{Houghton} introduced the concept of ends of the pair $(G,H)$, deriving several fundamental results and characterizing the case of exactly two ends under some assumptions. Scott \cite{Scott:Ends} rephrased the definition for discrete groups in terms of almost invariant subsets of the Schreier coset graph $Sch(\Gamma,\Lambda, S)$. Later, Sageev \cite{Sageev:EndsOfPairs} connected multi-endedness of a pair of groups to the existence of an essential action on a cubing.  Niblo and Roller \cite{NibloRoller} made use of this reinterpretation to show that for every subgroup $\Lambda$ of a countable property (T) group $\Gamma$, the Schreier graph $Sch(\Gamma,\Lambda,S)$ has at most one end. We will provide a more direct proof of this fact as part of Corollary \ref{cor:discrete}.

Let $X$ be a locally compact second countable (lcsc) topological space and fix a Radon measure $\mu$ on it. We say a space $X$ has \textbf{exactly one infinite volume end} if, after removing any compact subset $K$, the space $X-K$ has exactly one connected component with infinite $\mu$-measure.
For finitely generated discrete groups, the Cayley graph is one-ended if and only if it has exactly one infinite volume end with respect to the counting measure. 

Recall that for a discrete subgroup $\Gamma\leq G$ of a lcsc group $G$, we have the homogeneous space of left cosets $G/\Gamma$. In this setting, there exists a $\sigma$-finite Radon measure on $G/\Gamma$ denoted by $\mu_{G/\Gamma}$ that is nonsingular with respect to the action of $G$ by left multiplication (see \cite[Appendix B]{KazhdanBook}). If $G$ is unimodular, this measure can and will be taken to be $G$-invariant. In the nonunimodular case, the measure $\mu_{G/\Gamma}$ is only nonsingular and not at all canonical, so the notion of infinite volume end depends on the chosen representative in the measure class. Throughout, statements about infinite volume ends in the nonunimodular setting are always with respect to a fixed choice of nonsingular $\mu_{G/\Gamma}$. Groups with property (T) are all automatically unimodular. 

\subsection{Path connected groups}
We now present our first theorem, which, in the case of path connected unimodular lcsc groups, provides a result about the number of infinite volume ends of the homogeneous space $G/\Gamma$ for certain discrete subgroups $\Gamma$.

\begin{thm}\label{thm:1volumetricEndPC}
    Let $G$ be a path connected lcsc group and let $\Gamma\leq G$ be a discrete subgroup that is not coamenable. Assume that one of the following holds: 
    \begin{enumerate}
        \item $G$ has property (T). 
        \item $G=\langle G_n\rangle_{n\in\mathbf{N}}$ is generated by a sequence of subgroups such that for every $n\in \mathbf{N}$
        \begin{enumerate}
            \item $G_n$ commutes with $G_{n+1}$ and
            \item $L^2(G/\Gamma,\mu_{G/\Gamma})$ does not admit any nonzero $G_n$-invariant vectors and 
            \item there exists $i$ such that $L^2(G/\Gamma,\mu_{G/\Gamma})$ does not admit $G_i$-almost invariant vectors.
        \end{enumerate}
    \end{enumerate}
    Then the space $G/\Gamma$ has exactly one infinite volume end with respect to $\mu_{G/\Gamma}$.

    Moreover, if $G$ is unimodular and there is a positive lower bound for the injectivity radius of $G/\Gamma$, then $G/\Gamma$ has exactly one end.
\end{thm}

Corlette \cite{Corlette1990} previously showed that for convex cocompact subgroups $\Gamma$ of  $G=F_4^{(-20)}$ or $G=Sp(n,1)$ (for $n \geq 2$), the quotient space $G/\Gamma$ has at most one end. These cases are naturally covered by the property (T) version of Theorem \ref{thm:1volumetricEndPC}, since convex cocompact manifolds have a positive lower bound on their injectivity radius. More generally, Theorem \ref{thm:1volumetricEndPC} implies that if $\Gamma$ is an Anosov subgroup of a higher rank simple Lie group $G$, then $G/\Gamma$ has exactly one end, as does the locally symmetric space $K\backslash G/\Gamma$ where $K$ is a maximal compact subgroup of $G$. 

Condition (2b) which appears in the statement of Theorems \ref{thm:1volumetricEndPC} and \ref{thm:1VolumetricEndGraph} as well as Corollary \ref{cor:discrete} may seem slightly esoteric. 
If $\Gamma$ is a lattice of a product group $G_1\times \cdots \times G_k$, then $L^2_0(G/\Gamma, \mu_{G/\Gamma})$ admitting no $G_n$-invariant vectors is equivalent to $\Gamma$ being irreducible, in the sense that the projection $pr_n \Gamma\leq G_n$ is dense for every $1\leq n\leq k$.
Absent the assumption of finite covolume, irreducibility of $\Gamma$ still implies condition (2b).

\subsection{Symmetric spaces}

Theorem \ref{thm:1volumetricEndPC} has immediate implications for a wide range of geometric settings. In particular, it applies directly to symmetric spaces of noncompact type, as detailed in the following corollaries. For background about symmetric spaces and the terminology we use, see \cite{HelgasonBook}.

Denote by $\mathbf{H}_\mathbf{R}^n$ real hyperbolic $n$-space and by $\mathbf{H}_\mathbf{C}^n$ complex hyperbolic $n$-space. Since the isometry groups of the remaining irreducible symmetric spaces of noncompact type all have property (T), we have the following. 
\begin{cor}
    Let $X$ be a connected Riemannian symmetric space of noncompact type with no factors of type $\mathbf{H}_{\mathbf{R}}^n$ or $\mathbf{H}_\mathbf{C}^n$. Let $\Gamma\leq \mathrm{Isom(X)}$ be an infinite covolume discrete subgroup. 
    Then $\Gamma\backslash X$ has precisely one infinite volume end. 
\end{cor}

For locally symmetric manifolds $M=\Gamma\backslash X$ of the type above, ``one infinite-volume end'' means that after removing any compact set $K$, one component of $M\setminus K$ has infinite
Riemannian volume; the remaining noncompact components all must have finite volume and thus represent cusps. If there is a uniform lower bound on the injectivity radius (e.g., for Anosov/convex cocompact $\Gamma$), Lemma \ref{lem:injRadiusEndsAreVolumetric} identifies infinite volume ends with topological ends, so in this case $M$ is one ended as a manifold.

By applying case (2) of Theorem \ref{thm:1volumetricEndPC}, we obtain the following. 
\begin{cor}
    Let $X=X_1\times X_2$ be a product of connected Riemannian symmetric spaces of noncompact type. Assume $\Gamma\leq \mathrm{Isom}(X)$ is an irreducible discrete subgroup which is not coamenable. 
    Then $\Gamma\backslash X$ has precisely one infinite volume end. 
\end{cor}

\subsection{Groups acting on graphs}
Next, we consider the case where an lcsc group $G$ acts on a locally finite, connected graph $X$ by automorphisms with compact vertex stabilizers. This setup includes compactly generated totally disconnected lcsc groups, which act on their Cayley-Abels graphs \cite{KronMoller2007}. Additionally, affine buildings associated to reductive algebraic groups over local non-Archimedean fields provide further examples of such actions.

In this context, we compute the number of infinite volume ends of the quotient graph $X/\Gamma$ under identical rigidity assumptions as in the path-connected case. The measure on the graph $X/\Gamma$ is not, in general, the counting measure as one might expect. Let $v_0\in V(X)$ be a distinguished vertex. The map $\phi:G\to X$ given by $g\mapsto g. v_0$ descends to a map $\varphi: G/\Gamma\to X/\Gamma$ on the quotient spaces. The measure we use is the pushforward measure $\varphi_*\mu_{G/\Gamma}$. If $G$ is nonunimodular, the pushforward $\varphi_*\mu_{G/\Gamma}$ depends on the fixed choice of nonsingular measure $\mu_{G/\Gamma}$ and all infinite volume statements are with respect to that fixed choice.

\begin{thm} \label{thm:1VolumetricEndGraph}
    Let $G$ be an lcsc group and let $\Gamma\leq G$ be a discrete subgroup that is not coamenable. Assume $G$ acts on a connected locally finite graph $X$ by automorphisms with compact vertex stabilizers and the action on the vertices is transitive.
    
    Assume that one of the following holds: 
    \begin{enumerate}
        \item $G$ has property (T). 
        \item $G=\langle G_n\rangle_{n\in \mathbf{N}}$ is generated by a sequence of subgroups such that for every $n\in \mathbf{N}$
        \begin{enumerate}
            \item $G_n$ commutes with $G_{n+1}$ and
            \item $L^2(G/\Gamma,\mu_{G/\Gamma})$ does not admit any nonzero $G_n$-invariant vectors and 
            \item there exists $i$ such that $L^2(G/\Gamma,\mu_{G/\Gamma})$ does not admit $G_i$-almost invariant vectors.
        \end{enumerate}
    \end{enumerate}
    Then the graph $X/\Gamma$ has exactly one infinite volume end with respect to $\varphi_*\mu_{G/\Gamma}$.

    Moreover, if $G$ is unimodular and $\Gamma$ admits a finite index subgroup $\Gamma'\leq \Gamma$ such that $\Gamma'\acts X$ is free, then $X/\Gamma$ has exactly one end. 
\end{thm}

In related work, Bass and Lubotzky \cite{BassLubotzky2001} have studied extensively the situation where $X$ is a tree. This case is essentially orthogonal to the group $G$ having property (T) or being generated by commuting subgroups.

A direct application of Theorem \ref{thm:1VolumetricEndGraph} computes the number of ends of Schreier graphs.

\begin{cor}\label{cor:discrete}
    Let $\Gamma=\langle S\rangle$ be a countable discrete group with finite symmetric generating set $S$ and let $\Lambda\leq \Gamma$ be a subgroup which is not coamenable. 
    Assume that one of the following holds:  
    \begin{enumerate}
        \item $\Gamma$ has property (T). 
        \item $\Gamma=\langle \Gamma_n\rangle_{n\in\mathbf{N}}$ is generated by a sequence of subgroups such that for every $n\in\mathbf{N}$
        \begin{enumerate}
            \item $\Gamma_n$ commutes with $\Gamma_{n+1}$ and
            \item $\ell^2(\Gamma/\Lambda)$ does not admit any nonzero $\Gamma_n$-invariant vectors and 
            \item there exists $i$ such that $\ell^2(\Gamma/\Lambda)$ does not admit $\Gamma_i$-almost invariant vectors.
        \end{enumerate}
    \end{enumerate}
    Then the Schreier graph $Sch(\Gamma,\Lambda,S)$ has exactly one end. 
\end{cor}

Case (1) was previously proven through the work of Niblo and Roller \cite{NibloRoller} by applying the results of Sageev \cite{Sageev:EndsOfPairs}, but we provide an alternate direct proof. Note that outside of product groups, case (2) applies to certain discrete subgroups of any graph product of infinite groups arising from a connected graph. 

\medskip
\noindent {\bf Acknowledgements:} The author warmly thanks Sami Douba, Miko\l aj Fraczyk, and Robin Tucker-Drob  for several suggestions and many stimulating discussions. THe author also thanks an anononymous reviewer for helpful feedback. KW was supported by the Dioscuri program initiated by the Max Planck Society, jointly managed by the National Science Centre (Poland), and mutually funded by the Polish Ministry of Science and Higher Education and the German Federal Ministry of Education and Research.

\section{Preliminaries}

We work in the setting of lcsc topological spaces. In particular, such spaces admit an exhaustion by compact sets.
\subsection{Ends} Given an lcsc space $X$, fix an increasing sequence of compact sets $K_n$, for $n\in\mathbf{N}$, whose interiors cover $X$. 
An \textbf{end} is then a sequence of open sets $(U_n)_{n\in \mathbf{N}}$ such that $U_n$ is a connected component of $X-K_n$ and $U_{n+1}\subseteq U_n$ for all $n\in \mathbf{N}$. 
The \textbf{space of ends} $\mathcal{E}$ is then the set of all ends and comes with an ultrametric $d((U_i),(V_i))=2^{-\min\{i:U_i\not=V_i\}}$ that depends on the sequence $K_n$. 
Given two different compact exhaustions of $X$, the respective spaces of ends are easily shown to be homeomorphic, although the homeomorphism cannot be metric preserving in general. 

Assume now our lcsc space $X$ comes equipped with a Radon measure $\mu$ and a compact exhaustion $K_n$.
We say an end $(U_n)_{n\in \mathbf{N}}$ has \textbf{infinite volume} if $\mu(U_n)=\infty$ for all $n\in \mathbf{N}$.
The \textbf{space of infinite volume ends} $\mathcal{E}_\mu\subseteq \mathcal{E}$ is then the subset of ends consisting of infinite volume ends along with the subspace topology. Once more, since the measure is assumed Radon, up to homeomorphism the space of infinite volume ends $\mathcal{E}_\mu$ does not depend on the choice of compact exhaustion. By restricting ourselves to this class of ends we forget about regions that are entirely finite in measure, despite the fact that they may be topologically unbounded. Such finite regions are often called cusps in the study of locally symmetric spaces.

The only property of ends that we will use in our analysis is the following observation.
\begin{obs}\label{obs:ends}
    If $(X,\mu)$ has more than one infinite volume end, then there exists a set $A\subseteq X$ with compact boundary such that $\mu(A)=\mu(X-A)=\infty$.
\end{obs}

\subsection{1-Cohomology} Let $G$ be a topological group along with a continuous action $G\acts A$ on an abelian group. Then $Z^1(G, A)$ will denote the abelian group of 1-cocycles, i.e., continuous maps $c: G\to A$ satisfying the cocycle identity $c(gh)=c(g)+g. c(h)$. Let $B^1(G,A)$ denote the group of 1-coboundaries, i.e. maps $b: G\to A$ given by $b(g)=\xi-g .\xi$ for some $\xi\in A$.
Let $H^1(G,A)$ be the first cohomology group, which is defined as the quotient $H^1(G, A)\coloneqq Z^1(G,A)/B^1(G,A)$. If $\pi$ is a strongly continuous representation of $G$ into unitary operators $\cU(\cH)$ on a Hilbert space $\cH$, then we write $H^1(G,\pi)$ instead of $H^1(G,\cH)$ with the action $G\acts \cH$ being implemented through the representation $\pi$. 

Given a measure class preserving action $G\acts (X,\mu)$, recall that the associated \textbf{Koopman representation} $\kappa:G\to \mathcal{U}(L^2(X,\mu))$ is defined as $(\kappa(g)f)(x)=f(g^{-1}.x)(\tfrac{dg_*\mu}{d\mu})^{1/2}(x)$.

\begin{prop}\label{prop:almStabMapImpNonvanKoopCoh}
    Let $G$ be a topological group and let $G\acts(X,\mu)$ be a continuous measure class preserving transitive action on a $\sigma$-finite standard measure space $(X,\mu)$. Let $\kappa$ be the associated Koopman representation. If there exists a measurable set $B\subseteq X$ such that $\mu(B)$ and $\mu(X-B)$ are both infinite and $\mu(B\triangle gB)<\infty$ for every $g\in G$, then $H^1(G,\kappa)\not=\{0\}$.
\end{prop}

\begin{proof}
    Define a 1-cocycle $c: G\to L^2(X,\mu)$ by $c(g)=\mathbf{1}_B-g.\mathbf{1}_B$. By assumption, $c(g)\in L^2(X,\mu)$ for every $g\in G$ and $\mathbf{1}_B+a\not\in L^2(X,\mu)$ for any $a\in \mathbf{C}$. 
    
    Assume $c$ were a 1-coboundary, i.e., there exists $\xi\in L^2(X,\mu)$ such that $c(g)=\xi-g.\xi$. Then $\mathbf{1}_B-\xi=g.(\mathbf{1}_B-\xi)$ and, by transitivity, $\mathbf{1}_B+a=\xi\in L^2(X,\mu)$ for some $a\in \mathbf{C}$ which is a contradiction.
\end{proof}

When $G$ has property (T), it is a classical result that for every unitary representation $\pi: G\to\cU(\cH)$, the first cohomology group $H^1(G,\pi)=\{0\}$ vanishes \cite{Delorme1977}, and that this property characterizes property (T) \cite{Guichardet}. 
While we can't comment on every representation, if $G$ is generated by a sequence of chain-commuting groups, there is a sufficient condition for the first cohomology group with coefficients in a unitary representation to vanish. The following proof is essentially described in Peterson \cite{PetersonL2Rigid}, although it is not stated in this generality.

\begin{prop}\label{prop:commutingRigid}
    Let $G$ be a topological group and fix a strongly continuous unitary representation $\pi:G\to \cU(\cH)$. 

    Assume $G=\langle G_n\rangle _{n\in\mathbf{N}}$ is generated by subgroups $G_n$ such that, for every $n\in \mathbf{N}$,
    \begin{enumerate}
        \item $G_{n}$ commutes with $G_{n+1}$ and
        \item $\pi|_{G_{n}}$ does not admit any nonzero invariant vectors and
        \item there exists $i$ such that $\pi|_{G_i}$ does not weakly contain the trivial representation.
    \end{enumerate} 
    Then $H^1(G, \pi)=\{0\}$.
\end{prop}

\begin{proof}
Suppose $c:G\to \cH$ is a 1-cocycle. 
First, note that if $g,h\in G$ commute, then $\pi(g)c(h)-c(h)=c(gh)-c(g)-c(h)=\pi(h)c(g)-c(g)$.

Choose $i$ so that $\pi|_{G_i}$ also does not weakly contain the trivial representation. Then there exists a compact subset $Q\subseteq G_i$ and $C>0$ such that for every $\xi\in \cH$ we have 
\[
\|\xi\|_2\leq C\sup_{g\in Q}\|\pi(g)\xi-\xi\|_2.
\]
Hence, for every $h\in G_{i+1}$,
\[
    \|c(h)\|_2\leq C \sup_{g\in Q} \|\pi(g)c(h)-c(h)\|_2
    =C \sup_{g\in Q} \|\pi(h)c(g)-c(g)\|_2\leq 2C\sup_{g\in Q}\|c(g)\|_2
\]
    showing that $c|_{G_{i+1}}$ is bounded. By subtracting a 1-coboundary, we may thus assume $c|_{G_{i+1}}=0$. 
    \begin{claim}
        If $H$ and $K$ are commuting subgroups of $G$ and $c|_H=0$ and $\pi|_H$ does not admit any nontrivial invariant vectors, then $c|_K=0$.
    \end{claim} 
    \begin{proof}[Proof of Claim]  
        Note that $c(k)$ is $H$-invariant for every $k\in K$. Indeed, $\pi(h)c(k)=\pi(k)c(h)-c(h)+c(k)=c(k)$ for every $h\in H$. Therefore $c(k)=0$ for every $k\in K$.
    \end{proof}
    It follows from the claim that $c|_{G_n}=0$ for every $n\in\mathbf{N}$ and so $c=0$.
\end{proof}

\begin{remark}
     In the case where the group $G=G_0\times G_1$ is actually a product group, a stronger version of Proposition \ref{prop:commutingRigid} follows immediately from Shalom \cite[Corollary 1.8 and Proposition 3.2]{Shalom2000}. In that case, condition c) is superfluous since a representation $\pi$ of a product $G\times H$ weakly contains the trivial representation if and only if $\pi|_G$ and $\pi|_H$ both weakly contain the trivial representation. In particular, if the groups involved are products then c) is unnecessary in Proposition \ref{prop:commutingRigid}, and hence in Proposition \ref{prop:vanishingCoh}, Theorem \ref{thm:pc}, Theorem \ref{thm:graph}, and Corollary \ref{cor:discrete}.
\end{remark}

\subsection{Homogeneous spaces}
Given a discrete subgroup $\Gamma\leq G$ of a lcsc group $G$, the quotient space $G/\Gamma$ is called a \textbf{homogeneous space} and it comes with an action $G\acts G/\Gamma$ by left multiplication. 
The space $G/\Gamma$ admits a $\sigma$-finite Radon measure $\mu_{G/\Gamma}$ with the property that the action of $G$ preserves the measure class.  
If $G$ is unimodular, the measure $\mu_{G/\Gamma}$ will be taken to be $G$-invariant and in fact satisfy $\mu_G=\int_{G/\Gamma} g.\mu_\Gamma d\mu_{G/\Gamma}(g\Gamma)$ where $\mu_G$ is a fixed left-invariant Haar measure on $G$ and $\mu_\Gamma$ is the counting measure. See \cite[Appendix B]{KazhdanBook} for more about measures on homogeneous spaces.

Every lcsc group $G$ admits a left invariant proper metric $d_G$ that generates the topology \cite{HaagPrzy2006}. We then define the \textbf{injectivity radius at a point $x\in G/\Gamma$} as the supremum over $r\in \mathbf{R}$ such that the map $g\in G\mapsto gx$ is an injection when restricted to the ball of radius $r$. 
It is easy to see the \textbf{injectivity radius of $G/\Gamma$ is bounded away from 0} if and only if there exists a nonempty neighborhood of the identity $e\in U\subseteq G$ with compact closure such that for every $x\in G/\Gamma$ the map $g\in G\mapsto gx$ restricted to $U$ is injective. Hence, the property of having injectivity radius bounded away from 0 is independent of the choice of metric and we will make use of it without fixing a metric. 

Say an inclusion $\Gamma\leq G$ is \textbf{coamenable} if the Koopman representation $\kappa: G\to \cU(L^2(G/\Gamma,\mu_{G/\Gamma}))$  weakly contains the trivial representation. This is equivalent to the action $G\acts G/\Gamma$ being amenable in the sense of Greenleaf \cite{Greenleaf1969} and Eymard \cite{Eymard1972}.

\begin{prop}\label{prop:vanishingCoh}
    Let $\Gamma\leq G$ be a discrete subgroup of  a lcsc group $G$. Denote by $\kappa$ the Koopman representation of $G\acts(G/\Gamma,\mu_{G/\Gamma})$. Assume that one of the following holds: 
    \begin{enumerate}
        \item $G$ has property (T). 
        \item $G=\langle G_n\rangle_{n\in \mathbf{N}}$ is generated by a sequence of subgroups such that for every $n\in \mathbf{N}$
        \begin{enumerate}
            \item $G_n$ commutes with $G_{n+1}$ and
            \item $L^2(G/\Gamma,\mu_{G/\Gamma})$ does not admit any nonzero $G_n$-invariant vectors and
            \item there exists $i$ such that $L^2(G/\Gamma,\mu_{G/\Gamma})$ does not admit $G_i$-almost invariant vectors.  
        \end{enumerate}
    \end{enumerate}
    Then the first cohomology group $H^1(G,\kappa)=\{0\}$ vanishes.
\end{prop}

\begin{proof}
In case (1), property (T) implies Serre's property (FH) by \cite{Delorme1977} and thus $H^1(G,\kappa)=\{0\}$.

In case (2), we may apply Proposition \ref{prop:commutingRigid} to finish the proof.
\end{proof}

\section{Path connected groups}

For the duration of this section, let us restrict ourselves to the setting of homogeneous spaces of path connected groups. The example to keep in mind is that of homogeneous spaces of connected Lie groups as they are automatically path connected and due to a theorem of Yamabe \cite{Yamabe1953}, any connected locally compact group arises as a projective limit of Lie groups.

A version of the following proposition initially appeared in a paper of Freudenthal \cite{Freudenthal1931} where he proves it for the actions of a path connected group $G$ on itself by left and right translation. He makes use of it to prove that path connected groups can have at most two ends. 

\begin{lemma} \label{lem:cmpctBoundaryImpAlmStab}
    Let $G$ be a path connected topological group and let $G\acts X$ be an action by homeomorphisms on a topological space $X$. Let $A\subseteq X$ be a subset with compact boundary $\partial A$. Then $A\triangle gA$ is relatively compact for every $g\in G$.
\end{lemma}

\begin{proof}
    Fix $g\in G$. Since $G$ is path connected, there exists a path $\phi: [0,1]\to G$ such that $\phi(0)=e$ and $\phi(1)=g^{-1}$. Define the map $\Psi:[0,1]\times X\to X$ by $(r,x)\mapsto \phi(r)^{-1} x$ and note that $\Psi$ is continuous, so $\Psi([0,1]\times \partial A)$ is compact.

    Take $x\in A- gA$ and consider the path $\phi_{x}:[0,1]\to X$ defined by $r\mapsto \phi(r) x$. Now $\phi_{x}^{-1}(A)$ and $\phi_{x}^{-1}(X-A)$ are both nonempty so $\phi_{x}^{-1}(\partial A)$ is nonempty as well. We conclude $x\in\Psi([0,1]\times\partial A)$ and so $A-gA$ is relatively compact. 

    An identical argument shows $gA-A$ is also relatively compact.
\end{proof}

We now demonstrate that under an easy to satisfy assumption about injectivity radius, every end is an infinite volume end.

\begin{lemma}\label{lem:injRadiusEndsAreVolumetric}
    Let $\Gamma\leq G$ be a discrete subgroup of a unimodular lcsc group $G$. Assume that the injectivity radius of $G/\Gamma$ is bounded away from 0. 

    Then the space of ends $\mathcal{E}$ of $G/\Gamma$ is equal to the space of volumetric ends $\mathcal{E}_{\mu}$. 
\end{lemma}

\begin{proof}
    It is enough to show that for every compact set $K$ every non-relatively compact connected component $A$ of $X-K$ has infinite $\mu_{G/\Gamma}$-measure. 

    Since the injectivity radius is bounded away from 0, we may find a $U=U^{-1}\leq G$ which is a symmetric open neighborhood of the identity with compact closure such that for every $x\in G/\Gamma$ the map $g\in G\mapsto gx$ is injective when restricted to $U$. 
    By fixing a fundamental domain $F\subseteq G$ for the action by $\Gamma$ by right translation that intersects $U$ on a set of positive measure, we get that $\mu_G(U\cap F)>0$ is a lower bound on $\mu_{G/\Gamma}(Ux)$ for every $x\in G/\Gamma$. 
    Indeed,
    \[
    \mu_G(U\cap F)=\mu_G((U\cap F)g)=\mu_{G/\Gamma}((Ug\cap Fg)\Gamma)\leq\mu_{G/\Gamma}(Ug\Gamma).
    \]
    
    Fix compact $K\subseteq G/\Gamma$. Let $A$ be a non-relatively compact connected component of $G/\Gamma- K$. We will show $\mu_{G/\Gamma}(A)=\infty$. 
    By continuity of the action of $G$, the closures of $UK$ and $U^2x$ are compact for every $x\in G/\Gamma$. 
    Let $\{x_i\}_{i\in I}$ be a maximal collection of points in $A-UK$ satisfying $Ux_i\cap Ux_j=\emptyset$ for all $i\not=j$. 
    Observe the set $UK\cup\bigcup_{i\in I} U^2x_i$ covers all of $A$ since if there were a point $y\in A-\left (UK\cup\bigcup_{i\in I} U^2x_i\right )$ then $Uy\cap Ux_i=\emptyset$ for all $i\in I$ and $y\in A-UK$. 
    Since $A$ is not relatively compact, the collection $I$ must be infinite. Therefore, $\mu_{G/\Gamma}(A)\geq \mu_{G/\Gamma}(\bigcup_{i\in I} Ux_i)\geq \sum_{i\in I}\mu_G(U\cap F)=\infty$.
\end{proof}

We now prove the main theorem. 

\begin{thm}\label{thm:pc}
    Let $G$ be a path connected lcsc group and let $\Gamma\leq G$ be a discrete subgroup. Denote by $\kappa$ the Koopman representation of $G\acts (G/\Gamma,\mu_{G/\Gamma})$ and assume that the first cohomology group $H^1(G,\kappa)=\{0\}$ vanishes.
    Then the space $G/\Gamma$ has at most one infinite volume end.

    Moreover, if $G$ is unimodular and there is a positive lower bound for the injectivity radius of $G/\Gamma$, then $G/\Gamma$ has at most one end.
\end{thm}
\begin{proof}
    We prove the contrapositive, so assume $(G/\Gamma,\mu_{G/\Gamma})$ has more than one infinite volume end. Then by Observation \ref{obs:ends} there must be a set $A\subseteq G/\Gamma$ with compact boundary $\partial A$ satisfying $\mu_{G/\Gamma}(A)=\mu_{G/\Gamma}(G/\Gamma-A)=\infty$. Applying Lemma \ref{lem:cmpctBoundaryImpAlmStab}, we get that $A\triangle gA$ is relatively compact for every $g\in G$ and hence $\mu_{G/\Gamma}(A\triangle gA)$ is finite for every $g\in G$. It follows from Proposition \ref{prop:almStabMapImpNonvanKoopCoh} that in this case the first cohomology group $H^1(G,\kappa)$ does not vanish. 

    Lemma \ref{lem:injRadiusEndsAreVolumetric} now implies that if the injectivity radius is bounded away from 0, then the number of topological ends is the same as the number of infinite volume ends and thus bounded above by one.
\end{proof}

Theorem \ref{thm:1volumetricEndPC} now follows directly from Proposition \ref{prop:vanishingCoh} and Theorem \ref{thm:pc}.
We present one simple application to higher rank Lie groups. 

\begin{example}
    Let $G_1$ and $G_2$ be connected Lie groups with $G_1$ having property (T). Assume $\Gamma_1\leq G_1$ and $\Gamma_2\leq G_2$ are both Anosov subgroups with infinite covolume, then $(G_1\times G_2)/(\Gamma_1\times \Gamma_2)$ has one topological end. 

    Indeed, the fact that $\Gamma_1$ and $\Gamma_2$ are Anosov ensures a positive lower bound on injectivity radius. Then, by property (T), any infinite covolume subgroup $\Gamma_1\leq G_1$ is not coamenable. It follows that the action $G_1\acts (G_1\times G_2)/(\Gamma_1\times \Gamma_2)=G_1/\Gamma_1\times G_2/\Gamma_2$ is not amenable in the sense of Greenleaf and thus $\Gamma_1\times\Gamma_2\leq G_1\times G_2$ is not coamenable allowing us to apply the product case of Theorem \ref{thm:1volumetricEndPC}. 
\end{example}

\section{Groups acting on graphs}
We now move to the orthogonal setting of an lcsc group $G$ acting by automorphisms on a locally finite connected graph $X$ with compact vertex stabilizers. Examples of such actions include compactly generated totally disconnected locally compact groups acting on their Cayley-Abels graphs and non-Archimedean Lie groups acting on the 1-skeleton of their Bruhat-Tits building. Given a discrete subgroup $\Gamma\leq G$, we investigate the number of ends and volumetric ends of the quotient graph $X/\Gamma$. The results and proofs mirror those regarding homogeneous spaces of path connected groups. 

\textbf{Ends of a graph.} Given a connected graph $X$, we can think of $X$ as a path connected metric space by replacing each edge with an isometric copy of the interval $[0,1]$. The space of ends of a connected graph $X$ is then defined to be the space of ends of $X$ as a path connected topological space. 

Graph theoretic ends have traditionally been studied as equivalence classes of rays to infinity. In general, these graph theoretic and topological notions differ, however, when the graph is locally finite there is a one to one correspondence between the space of graph theoretic ends and the space of topological ends as defined above. 

\begin{remark}
    In the context where $X$ is a tree and $G$ is $\Aut(X)$, the graph $X/\Gamma$ has been studied extensively in the book of Bass and Lubotzky \cite{BassLubotzky2001}. Let $v_0\in V(X)$ be a distinguished vertex. The map $\phi:G\to X$ given by $g\mapsto g. v_0$ descends to a map $\varphi: G/\Gamma\to X/\Gamma$ on the quotient spaces. Bass and Lubotzky have examples where the pushforward of the Haar measure $\varphi_*\mu_{G/\Gamma}$ is finite but yet the graph $X/\Gamma$ has infinitely many ends.
\end{remark}

\begin{lemma}\label{lem:cutImpliesCohGraph}
    Let $G$ be an lcsc group and let $\Gamma\leq G$ be a discrete subgroup. Let $X$ be a connected locally finite graph. 
    Fix an action $G\acts X$ by graph automorphisms with compact vertex stabilizers that acts transitively on the vertices.

    Assume that there exists a set $A\subseteq V(X)/\Gamma$ with finite edge boundary. 
    Then $B\coloneqq\varphi^{-1}(A)\subseteq G/\Gamma$ is such
    that $B\triangle gB$ is relatively compact for every $g\in G$.

    Moreover, if $G$ is unimodular and $\Gamma$ admits a finite index subgroup $\Gamma'\leq \Gamma$ such that the restricted action $\Gamma'\acts V(X)$ is free and $A$ is an infinite coinfinite set, then $B$ satisfies $\mu_{G/\Gamma}(B)=\mu_{G/\Gamma}(X-B)=\infty$. 
\end{lemma}

\begin{proof}
    Let $q:X\to X/\Gamma$ and $p:G\to G/\Gamma$ denote the quotient maps. Assume $A\subseteq V(X)/\Gamma$ is a set with finite edge boundary.  Fix a section $r: X/\Gamma\to X$ of the quotient $q$ and choose a distinguished vertex $v_0\in V(X)$. 

    For each $g\in G$, fix a path $P_g$ between $v_0$ and $g. v_0$. 
    Since the action $G\acts X$ is by automorphisms, $h.P_g$ is a path of length $|P_g|$ between $h.v_0$ to $hg.v_0$ for every $h\in G$. It follows that $q(h.P_g)$ is a path between $q(h.v_0)$ and $q(hg.v_0)$ of length at most $|P_g|$. 

    Define $E\coloneqq \{k\in G: \exists \gamma\in \Gamma \text{ such that }\gamma k^{-1}. v_0\in r(A)\}$. Note that $E$ is invariant under the action of $\Gamma$ by right multiplication. The sets thus $p(E)=\varphi^{-1}(A)=B$ coincide.

    We first show that $B\triangle gB$ is relatively compact for every $g\in G$. Since $E$ is $\Gamma$-invariant, $B\triangle gB=p(E\triangle gE)$. 
    The argument resembles the argument we used in the path connected setting. 
    Assume that $h\in E-gE$. 
    So there is a $\gamma\in \Gamma$ such that $\gamma h^{-1}.v_0\in r(A)$ which implies $q(h^{-1}.v_0)=q(\gamma h^{-1}.v_0)\in A$.
    Similarly, for every $\gamma\in \Gamma$ we have $\gamma  h^{-1}g.v_0\not \in r(A)$ and thus $q(h^{-1}g.v_0)\not \in A$. The path $q(h^{-1}.P_{g})$ therefore intersects $\partial A$ and witnesses that the distance between $q(h^{-1}. v_0)$ and $\partial A$ is at most $|P_{g}|$. It follows that there is a $\gamma\in \Gamma$ such that the distance between $\gamma h^{-1}. v_0$ and $r(\partial A)$ is at most $|P_{g}|$.

    The set $r(\partial A)$ is a finite set, and hence by local finiteness and compactness of the stabilizers the set $C_g\coloneqq\{k\in G: k^{-1}. v_0\in B_{|P_{g}|}(r(\partial A))\}$ is compact. By the previous paragraph $E-gE\subseteq C_g\Gamma$ and hence $B\triangle gB=p(E-gE)\subseteq p(C_g)$ lives in a continuous image of a compact set. Repeating mutatis mutandis allows us to conclude that $B\triangle gB$ is relatively compact for every $g\in G$.

    Now assume that there exists finite index subgroup $\Gamma'\leq \Gamma$ such that the restricted action $\Gamma'\acts V(X)$ is free. Assume further that $A$ is both infinite and coinfinite. Observe that $r(A)$ and $r(A^C)$ are infinite disjoint sets and the disjoint union $r(A)\sqcup r(A^C)$ is a fundamental domain for $\Gamma\acts V(X)$ where $A^C=V(X)/\Gamma-A$. 
    
    Fix a set of right coset representatives $S$ for the inclusion $\Gamma'\leq \Gamma$. 
    Then the disjoint union $\bigsqcup_{s\in S}s. r(A)\sqcup \bigsqcup_{s\in S} s .r(A^C)$ is a fundamental domain for the action $\Gamma'\acts X$.
    Define $D_0\coloneqq\{g\in G: g^{-1}. v_0\in \bigsqcup_{s\in S} s .r(A)\}$ and $D_1\coloneqq \{g\in G: g^{-1}. v_0\in \bigsqcup_{s\in S}s. r(A^C)\}$. We proceed by computing the measures of $D_0$ and $D_1$. 

    Note that the measure of the stabilizer $\mu_G(G_{v_0})>0$ is nonzero since $G$ can be written as a countable union of left translates of $G_{v_0}$. Since $r(A)$ is infinite, $D_0$ can be written as a disjoint union of infinitely many left translates of $G_{v_0}$. Since $\mu_G$ is left-invariant, $\mu_G(D_0)=\infty$. Similarly, $\mu_G(D_1)=\infty$. Since $\Gamma'\acts V(X)$ is free, the set $D_0\sqcup D_1$ is a fundamental domain for $\Gamma' \acts G$ by right multiplication. 

    Since the inclusion $\Gamma'\leq \Gamma$ is finite index, there exists $D\subseteq D_0$ with $[\Gamma:\Gamma']\mu_G(D)=\mu_G(D_0)=\infty$ such that $D$ is a fundamental domain for the action $\Gamma\acts D_0\Gamma'$ by right multiplication. Similarly, there is a $\tilde D\subseteq D_1$ with $[\Gamma:\Gamma']\mu_G(\tilde D)=\mu_G(D_1)=\infty$ such that $\tilde D$ is a fundamental domain for $\Gamma\acts D_1\Gamma'$. Observe that $B=p(D)$ and $G/\Gamma-B=p(\tilde D)$ so $\mu_{G/\Gamma}(B)=\mu_{G/\Gamma}(G/\Gamma-B)=\infty$ by the properties of $\mu_{G/\Gamma}$ finishing the proof.
\end{proof}

\begin{thm}\label{thm:graph}
    Let $G$ be an lcsc group and let $\Gamma\leq G$ be a discrete subgroup. Assume $G$ acts transitively on the vertices of a connected locally finite graph $X$ by automorphisms with compact vertex stabilizers.
    
    Let $\kappa$ denote the Koopman representation of $G\acts (G/\Gamma,\mu_{G/\Gamma})$ and assume that the first cohomology group $H^1(G,\kappa)=\{0\}$ vanishes.
    Then the graph $X/\Gamma$ has at most one infinite volume end with respect to $\varphi_*\mu_{G/\Gamma}$.

    Moreover, if $G$ is unimodular and $\Gamma$ admits a finite index subgroup $\Gamma'\leq \Gamma$ such that $\Gamma'\acts X$ is free, then $X/\Gamma$ has at most one end. 
\end{thm}

\begin{proof}
    The theorem then follows from Observation \ref{obs:ends}, and Proposition \ref{prop:almStabMapImpNonvanKoopCoh}, and Lemma \ref{lem:cutImpliesCohGraph}. 
\end{proof}

Theorem \ref{thm:1VolumetricEndGraph} now follows directly from Proposition \ref{prop:vanishingCoh} and Theorem \ref{thm:graph}.

\bibliographystyle{alpha}
\bibliography{myReference}

\end{document}